\documentclass[graybox]{svmult}
\usepackage{mathptmx}       
\usepackage{helvet}         
\usepackage{courier}        
\usepackage{makeidx}         
\usepackage{graphicx}        
\usepackage{multicol}        
\usepackage[bottom]{footmisc}
\usepackage{amsmath,amssymb,amsfonts}        
\usepackage{ulem}
\usepackage{enumitem}
\newcommand{\ecr}[1]{\textcolor{red}{#1}}
\newcommand{\YR}[1]{\textcolor{blue}{#1}}
\makeindex             

\begin{document}
\title*{Low-dimensional indecomposable representations of the braid group $B_3$}
\author{Eric C. Rowell and Yuze Ruan}

\institute{Eric C. Rowell \at Texas A\&M University, College Station, TX 77843,
USA, \email{rowell@tamu.edu}
\and Yuze Ruan \at Tsinghua University,  Beijing, 100084, China,
 \email{ryz22@mails.tsinghua.edu.cn}}
%
%

\maketitle

\abstract{In this note we give a complete classification of all indecomposable yet reducible representations of $B_3$ for dimensions $2$ and $3$ over an algebraically closed field $K$ with characteristic $0$, up to equivalence.  We illustrate their utility with an example.}

\section{Introduction}

The problem of classifying representations of the braid groups $B_n$ is unsurprisingly difficult, as is typically the case for infinite non-abelian discrete groups.  On the other hand, even for small $n$ and in low dimensions classifications can be invaluable in applications to 
 braided tensor categories \cite{ruma}, especially in questions relevant to quantum computation \cite{LRW,RW18}. Indeed, the classification of irreducible $B_3$ representations of dimension $\leq 5$ by Tuba and  Wenzl \cite{TW01} gave purchase on questions in \emph{loc. cit.}. In particular their classification shows that such irreducible representations are determined, up to finite ambiguity, by the eigenvalues of the image of a standard generators of $B_3$, and conversely show that irreducible representations always exist for a given list of eigenvalues, provided certain polynomials are non-vanishing.
 
 It is natural to consider indecomposable representations, which play a similar role in non-semisimple settings. 
In this paper we give a complete classification of all indecomposable but not irreducible representations (\textit{strictly indecomposable}) of $B_3$ for dimension up to 3 over an algebraic closed field $K$ of characteristic $0$. We find that there are many more forms a strictly indecomposable representation of $B_3$ may take, which is more complicated than the result of simple representations given by Wenzl and Tuba. More precisely, let $\sigma_1, \sigma_2$ be the generators of $B_3$, and $A$ is the linear endomorphism by which $\sigma_1$ acts on some representation $V$.  If $V$ is simple with dimension less than 3, it is uniquely determined up to equivalence by the eigenvalues of $A$, whereas for strictly indecomposable representations there may be further parameters and choices. In particular, in some cases there are relations among the eigenvalues of $A$, and there may be inequivalent indecomposable representations with the same $A$.
Our approach in this paper is computational and elementary: we first reduce $A$ to the Jordan normal form, then determine all possibilities for the common invariant subspace, then compute the solutions, finally we verify all the solutions are indeed inequivalent and indecomposable.
In the last section of this paper, we will also discuss an application and outline ways to extend our results.  The appendix contains details of the computer calculations. \\   

\textbf{Acknowledgments} The computations for this research were largely carried out while Y.R. was an undergraduate in Mathematics at Texas A\&M University, and the hospitality of that institution is gratefully acknowledged.  The research of E.C.R. is partially supported by US NSF grant DMS-2205962. Y.R. is supported by  Beijing Municipal Science \& Technology
Commission Z221100002722017.

\section{Preliminaries}
\begin{definition}
The Artin braid group $B_n$ is the group generated by $\sigma_ 1,\sigma_2,\ldots,\sigma_{n-1}$ satisfying relations $\sigma_{i}\sigma_{j} = \sigma_{j}\sigma_{i}$ if $|i-j|\geq 2$, and\ $$ \sigma_{i}\sigma_{i+1}\sigma_{i} = \sigma_{i+1}\sigma_{i}\sigma_{i+1},\  for\  1\leq i\leq n-2.$$
The group $B_3$ which is generated by two elements $\sigma_ 1,\sigma_2$ satisfying the relation $ \sigma_{1}\sigma_{2}\sigma_{1} = \sigma_{2}\sigma_{1}\sigma_{2}$.
\end{definition}

We refer to \cite{KT08} for more details on the structures of the braid group.

\begin{definition}{\textbf{Direct sums and indecomposable representations}} Let $G$ be a group.
If $(V,\varphi)$ and $(W,\rho)$ are both $G$-representations then:
\begin{enumerate}[label=\upshape(\alph*)
]
    \item  ($V\oplus W,(\phi,\rho)$) is a $G$-representation.
    \item if $(V,\varphi)=(W_1\oplus W_2,(\rho_1,\rho_2))$, and $W_1\neq V, W_1\neq 0$, then $(V,\varphi)$ is \textbf{decomposable}. Otherwise, it is said to be \textbf{indecomposable}.
    \item $(V,\varphi)$ is \textbf{irreducible} if the only $\rho(G)$ invariant subspaces of $V$ are $0$ and $V$.
    \item We call a representation \textbf{strictly indecomposable} if it is indecomposable but not irreducible.
\end{enumerate}

\end{definition}

In what follows, for a $B_3$ represenation $(\rho,V)$ we set $\rho(\sigma_1)=A$ and $\rho(\sigma_2)=B $. Since $B_3$ is generated by $\sigma_1, \sigma_2$, we have
\begin{lemma}
The representation $(V,\rho)$ is reducible if and only if there exists a proper non-zero subspace $W\subset V$ which is an invariant subspace of both $A$ and $B$, and $(V,\rho)$ is indecomposable iff there are no proper non-zero invariant subspaces $W_i$ such that $V= W_1\oplus W_2$.
\end{lemma}

The following is synthesized from \cite{TW01}, focusing on the $2$ and $3$ dimensional cases.
\begin{theorem}[\cite{TW01}]\label{TWthm23}
 Any $d=2$ or $d=3$ dimensional irreducible representation $\rho$ of $B_3$ over $K$ of is equivalent to one of the form:
\[A=\begin{pmatrix}
\lambda_1&\lambda_1\\
0&\lambda_2
\end{pmatrix},\ \ \ \ \  B=\begin{pmatrix}
\lambda_2&0\\
-\lambda_2&\lambda_1
\end{pmatrix}\ \ \ \ \   \text{for}\  \ d=2,\]

\[A=\begin{pmatrix}
\lambda_1&\lambda_1\lambda_3\lambda_2^{-1}+\lambda_2&\lambda_2\\
0&\lambda_2&\lambda_2\\
0&0&\lambda_3\\
\end{pmatrix},\ \ \ \ \   B=\begin{pmatrix}
\lambda_3&0&0\\
-\lambda_2&\lambda_2& 0\\
\lambda_2&-\lambda_1\lambda_3\lambda_2^{-1}-\lambda_2&\lambda_1
\end{pmatrix}\ \ \ \ \   \text{for}\  \ d=3.\]
where $\lambda_i\in K^\times$.
Moreover, these representations {are} irreducible if and only if the $\lambda_i$ do not satisfy $\lambda_1^2-\lambda_1\lambda_2+\lambda_2^2=0$ for $d=2$ or $(\lambda_1^2+\lambda_2\lambda_3)(\lambda_2^2+\lambda_1\lambda_3)(\lambda_3^2+\lambda_1\lambda_2)=0$ for $d=3$.
\end{theorem}

It will be useful to assume that $A$ is in Jordan form, so that the standard basis forms a basis of generalized eigenvectors. Since any $A$-invariant subspace is spanned by generalized eigenvectors, we have:

\begin{lemma}[\cite{TW01}]\label{lem:bi-ei}
 If $\{e_i\}_{1\leq i\leq n}$ is a basis of generalized eigenvectors of $A$ then $\{b_i\}_{1\leq i\leq n}$ is a basis of generalized eigenvectors of $B$ where $b_i=ABAe_i$.
\end{lemma}
\begin{proof}
This follows from the fact that $B$ is obtained from conjugating $A$ by $ABA$.
\end{proof}

\section{Strictly Indecomposable Representations}
From now on, we assume $(V,\rho)$ is a strictly indecomposable  $n$-dimensional $B_3$ representation over an algebraically closed field $K$ of characteristic $0$ and $W$ is an invariant space for $A$ and $B$ of minimal dimension. 
\begin{lemma}\label{lem:basis} 
Suppose  $\{e_k\}_{1\leq k\leq n}$ is a basis of generalized eigenvectors of $A$. If $W=span\{e_i\}_{i\in I}$ for some $I\subset \{1,\ldots,n\}$ is a invariant subspace for $A$ and $B$, then we also have $W=span\{b_i:=ABAe_i\}_{i\in I}$.
\end{lemma}
\begin{proof}
Since $W$ is also invariant under $ABA$, and $ABA$ is invertible, we have $ABA(W)=W$. Thus by Lemma \ref{lem:bi-ei}, we see $b_i\in W$ for all $i\in I$ and they form a basis of $W$.
\end{proof}

\begin{theorem}\label{thm:n=2}
If $\dim(V)=2$ then any strictly indecomposable representation is equivalent to one of the following:
\begin{enumerate}[label=\upshape(\arabic*)
]
    \item \label{n=2_first}
$A=\begin{pmatrix}
\lambda_1&0\\
0&\lambda_2
\end{pmatrix}\ \ \ \ \ 
B=\begin{pmatrix}
\lambda_1&1\\
0&\lambda_2
\end{pmatrix},\ \ \ \ \ \lambda_2=\lambda_1 e^{\pm\pi i/3}$

    \item \label{n=2_second}
$A=B=\begin{pmatrix}
\lambda&1\\
0&\lambda 
\end{pmatrix}.
$
\end{enumerate}

\end{theorem}
\begin{remark}
    Note that by Theorem \ref{TWthm23} there are no irreducible 2-dimensional representations with eigenvalues as in (1).  While the eigenvalues of (2) can appear in an irreducible representation, this representation is clearly reducible.
\end{remark}

\begin{proof}

In this case the minimal invariant subspace $W$ has $\dim(W)=1$, so that $A$ and $B$ both act on $W$ by the same scalar $\lambda$. We may assume that $A$ is in Jordan form which gives us two cases:
\begin{enumerate}[label=\upshape \textbf{Case\ \arabic*}, itemindent=3.5em, leftmargin=0pt]
 \item \label{case:n=2_first}: If $A=\begin{pmatrix}
\lambda_1&0\\
0&\lambda_2
\end{pmatrix}$ then the standard basis vectors $e_1,e_2$ are eigenvectors, so any $1$-dimensional $A,B$ invariant subspace is spanned by either $e_1$ or $e_2$. Thus, by Lemma \ref{lem:basis}, $W=span\{e_i\}=span\{b_i\},$ for either $ i=1$  or $i=2$. Without loss of generality we may assume $W=span\{e_1,b_1\}$.  Thus, since $b_1=ABAe_1$ we have that $b_1=\lambda_1^3e_1$.  Set $b_2=\gamma_1 e_1+\gamma_2 e_2$.  By indecomposability, $\gamma_1\neq 0$, also $\gamma_2\neq 0$, by Lemma \ref{lem:basis}. Then by Lemma \ref{lem:bi-ei}, $ABA=\begin{pmatrix}
\lambda_1^3&\gamma_1\\
0&\gamma_2
\end{pmatrix}
$. 
Let $M=\begin{pmatrix}
\gamma_1&0\\
0&1
\end{pmatrix}
$. Since $A$ commutes with $M$, by conjugation with $M$ we get an equivalent representation with $ABA=\begin{pmatrix}
\lambda_1^3&1\\
0&\gamma_2
\end{pmatrix}$. Then by computation, we get $B$ by solving the equation $ABA=BAB$ with three unknown variables: $\gamma_2,\lambda_1,\lambda_2$.  We find $\gamma_2=\lambda_2^3$ where  $\lambda_2^2-\lambda_1\lambda_2+\lambda_1^2=0$, and conclude $B=\begin{pmatrix}
\lambda_1&\frac{1}{\lambda_1\lambda_2}\\
0&\lambda_2
\end{pmatrix}$. Then conjugating by $\begin{pmatrix}
\lambda_1\lambda_2&0\\
0&1
\end{pmatrix}$, we obtain the equivalence class represented by \ref{n=2_first}.

\item \label{case:n=2_second}: 
$A=\begin{pmatrix}
\lambda&1\\
0&\lambda
\end{pmatrix}$, then by similar argument, we get  $ABA=\begin{pmatrix}
\lambda^3&\gamma_1\\
0&\gamma_2
\end{pmatrix}$.  In this case we cannot conjugate any variables away, and we solve the braid equation with three unknown variables  $\gamma_1,\gamma_2,\lambda$ directly, obtaining the equivalence class represented by \ref{n=2_second}.
\end{enumerate} 

Clearly these classes are distinct since the Jordan normal forms of $A$ are different in these cases. They are also indecomposable, since for \ref{n=2_first} the only way to write $V$ as a direct sum of proper invariant subspaces of $A$ is $Ke_1\oplus Ke_2$, but $Ke_2$ is not invariant for $B$. It is also obvious for \ref{n=2_second}.

\end{proof}

Next we aim to classify 3-dimensional strictly indecomposable representations of $B_3$. The strategy is similar to that of Theorem \ref{thm:n=2}, which is outlined as follows:
\begin{enumerate}[label=\upshape \textbf{Step\ \arabic*}, itemindent=3.5em, leftmargin=0pt]

\item\label{step1} Reduce $A$ to the Jordan normal form.

\item\label{step2} Find the possible basis for $W$ and write down $ABA$ with some extra unknown variables with certain conditions.

\item\label{step3} Compute $B$ (which is easy since $A^{-1}$ has a very simple form) and by using Maple \cite{Maple} (see Appendix) we can solve the equation: $ABA=BAB$, which ensures that we have a $B_3$ representation, although they may fail to be strictly indecomposable.

\item\label{step4} Verify indecomposability and deduce inequivalence: First determine all the forms of matrices that commute with $A$ (when $A$'s characteristic polynomial equals its minimal polynomial, such matrices are polynomials of $A$) and conjugate $B$ by them.  Thus we can simplify these solutions to guarantee distinctness.  Finally we can deduce whether the representation is indecomposable or not by checking finite possibilities of direct sum decomposition  (see Appendix).  
\end{enumerate}

\begin{theorem}\label{thm:n=3_W=1}
when $n=3,\  \dim(W)=1$, every strictly indecomposable representation has one of the following forms (up to equivalence), where $\lambda_i$ are mutually different and non-zero:
\begin{enumerate}[label=\upshape(\arabic*)
]
\item 
$\ A=\begin{pmatrix}
\lambda_1&0&0\\
0&\lambda_2&0\\
0&0&\lambda_3
\end{pmatrix}$

\begin{enumerate}
[label=\upshape(1.\arabic*)
]
\item \label{n=3_W=1_A1B1}
$B=\begin{pmatrix}
\lambda_1&1&1\\
0&\lambda_2&0\\
0&0&\lambda_3
\end{pmatrix}$\ \ \ \ \ \  ($\lambda_2=\lambda_1 e^{\pm\pi i/3},\ \lambda_3=\lambda_1 e^{\mp\pi i/3})
$
\item \label{n=3_W=1_A1B2}
$B=\begin{pmatrix}
\lambda_1&0&1\\
0&\lambda_2&1\\
0&0&\lambda_3
\end{pmatrix}$\ \ \ \ \ \  ($\lambda_1=\lambda_3 e^{\pm\pi i/3},\ \lambda_2=\lambda_3 e^{\mp\pi i/3} )
$
\item \label{n=3_W=1_A1B3}
$B=\begin{pmatrix}
\lambda_1&1&1\\
0&\lambda_2&2\lambda_2\\
0&0&\lambda_3
\end{pmatrix}$\ \ \ \ \ ($\lambda_1=\lambda_2 e^{\pm\pi i/3},\ \lambda_3=\lambda_2 e^{\mp\pi i/3}) 
$
\item \label{n=3_W=1_A1B4}  
$B=\begin{pmatrix}
\lambda_1&1&\frac{-\lambda_2^2(\lambda_1^2-\lambda_1\lambda_2+\lambda_2^2)}{\lambda_1(\lambda_1^2+\lambda_2^2)}\\
0&\frac{-\lambda_1^4}{\lambda_2(\lambda_1^2+\lambda_2^2)}&\frac{\lambda_1^2(\lambda_1^4+\lambda_1^2\lambda_2^2+\lambda_2^4)}{(\lambda_1^2+\lambda_2^2)^2}\\
0&1&\frac{\lambda_2^3}{\lambda_1^2+\lambda_2^2}
\end{pmatrix}$\ \ \ \  ($\lambda_3=-\frac{\lambda_1^2}{\lambda_2}$)\\\\\\
\end{enumerate}

\item \label{n=3_W=1_A2B}
$\ A=\begin{pmatrix}
\lambda_1&0&0\\
0&\lambda_1&0\\
0&0&\lambda_2
\end{pmatrix}$\ \ \ \ \  
$B=\begin{pmatrix}
\lambda_1&1&\frac{\lambda_1\lambda_2}{\lambda_1-\lambda_2}\\
0&\frac{\lambda_2^2}{\lambda_2-\lambda_1}&0\\
0&1&\frac{\lambda_1^2}{\lambda_1-\lambda_2}
\end{pmatrix}$
\ \ \ \ ($\lambda_1=\lambda_2 e^{\pm\pi i/3}$)

\item 
$\ A=\begin{pmatrix}
\lambda_1&0&0\\
0&\lambda_2&1\\
0&0&\lambda_2
\end{pmatrix}$

\begin{enumerate}[label=\upshape(3.\arabic*)
]
\item \label{n=3_W=1_A3B1}
$B=\begin{pmatrix}
\lambda_1&0&1\\
0&\lambda_2&1\\
0&0&\lambda_2
\end{pmatrix}$
\ \ \ \ ($\lambda_1=\lambda_2e^{\pm\pi i/3}$)

\item \label{n=3_W=1_A3B2}
$B=\begin{pmatrix}
\lambda_1&1&\frac{\lambda_1+\lambda_2}{\lambda_1^2}\\
0&\lambda_2&0\\
0&\lambda_1^2&\lambda_2
\end{pmatrix} $\ \ \ \ ($\lambda_2=\lambda_1 e^{\pm\pi i/2}$)
\end{enumerate}

\item $\ A=\begin{pmatrix}
\lambda_1&1&0\\
0&\lambda_1&0\\
0&0&\lambda_2
\end{pmatrix}$

\begin{enumerate}[label=\upshape(4.\arabic*)
]
\item \label{n=3_W=1_A4B1}
$B=\begin{pmatrix}
\lambda_1&1+\frac{\lambda_1\beta}{\lambda_2^2}&1\\
0&\lambda_1&0\\
0&\beta&\lambda_2
\end{pmatrix}$\ \ \ ($\lambda_1=\lambda_2 e^{\pm\pi i/3}$, $\beta$ an arbitrary parameter.)

\item \label{n=3_W=1_A4B2}
$B=\begin{pmatrix}
\lambda_1&-2&0\\
0&\frac{\lambda_2}{2}&1\\
0&\frac{3\lambda_2^2}{4}&-\frac{\lambda_2}{2}
\end{pmatrix} $\ \ \ \ ($\lambda_1=-\lambda_2$)
\end{enumerate}

\item \label{n=3_W=1_A5B}
$\ A=B=\begin{pmatrix}
\lambda&1&0\\
0&\lambda&1\\
0&0&\lambda
\end{pmatrix}$  
\end{enumerate}
\end{theorem}
\begin{remark}
    Notice that in all of the cases listed in (1)  as well as (3.2) and (4.2) no representation with these eigenvalues can be irreducible, as the polynomial $(\lambda_1^2+\lambda_2\lambda_3)(\lambda_2^2+\lambda_1\lambda_3)(\lambda_3^2+\lambda_1\lambda_2)$ vanishes.  Note also that in case (4.1) the equivalence class of representations is not determined up to finite choices by the eigenvalues--there is an additional parameter $\beta$.  
\end{remark}
\begin{proof}
The proof splits into cases depending on the  Jordan normal forms of $A$ (\ref{step1}).

\begin{enumerate}[label=\upshape \textbf{Case\ \arabic*}, itemindent=3.5em, leftmargin=0pt]
 \item \label{case:n=3_W=1_case1}:
 $A$ is a diagonal matrix with three distinct eigenvalues. By Lemma \ref{lem:basis}, we may assume $W=span\{e_1\}=span\{b_1\}$, then $ABA=\begin{pmatrix}
\lambda^3_1&\beta_1&\gamma_1\\
0&\beta_2&\gamma_2\\
0&\beta_3&\gamma_3
\end{pmatrix}$ with $\beta_1,\gamma_1$ not equal to 0 simultaneously (by indecomposability). Using Maple, we get \ref{n=3_W=1_A1B1}-\ref{n=3_W=1_A1B4}.

\item \label{case:n=3_W=1_case2}: 
$A$ is a diagonal matrix with two distinct eigenvalues. We may assume $W_1=span\{e_1,e_2\}$ is a two-dimensional eigenspace for $A$ with eigenvalue $\lambda_1$. Then by Lemma \ref{lem:bi-ei}, $W_2=span\{b_1,b_2\}$ is a two-dimensional eigenspace for $B$ with eigenvalue $\lambda_1$.  Since $V$ is 3-dimensional, we have $W_0=W_1\cap W_2\neq 0$. If $\dim(W_0)=2$, then $Be_i=\lambda_1 e_i$, which implies $B=\begin{pmatrix}
\lambda_1&0&\gamma_1\\
0&\lambda_1&\gamma_2\\
0&0&\lambda_2
\end{pmatrix}$ with $\gamma_1,\gamma_2\neq 0$. 
Then conjugating by $M=\begin{pmatrix}
\gamma_1&0&0\\
0&\gamma_2&0\\
0&0&1
\end{pmatrix}$ we can assume $\gamma_1,\gamma_2$ both to be $1$.   Solving the braid equation to get the relation for eigenvalues, we find the equivalence class is represented by
$B=\begin{pmatrix}
\lambda_1&0&1\\
0&\lambda_1&1\\
0&0&\lambda_2
\end{pmatrix}$,  However if we let 
$M_1=\begin{pmatrix}
1&1&0\\
2&1&0\\
0&0&1
\end{pmatrix}$, then we have $M^{-1}_1 AM_1=A,\  M^{-1}_1 BM_1=\begin{pmatrix}
\lambda_1&0&0\\
0&\lambda_2&1\\
0&0&\lambda_2
\end{pmatrix}$, from which we conclude representation is decomposable. 

If $\dim(W_0)=1$ then there are scalars $ t_1,t_2,k_1,k_2\in K$ such that $ t_1e_1+t_2e_2=k_1b_1+k_2b_2$.  Without loss of generality, we may assume $t_1\neq 0$, and change the basis via $(\widetilde{e_1},\widetilde{e_2})=(e_1,e_2)\begin{pmatrix}
t_1&0\\
t_2&1
\end{pmatrix}
$. Replacing $A$ and $B$  by their conjugates by $M=\begin{pmatrix}
t_1&0&0\\
t_2&1&0\\
0&0&1
\end{pmatrix}
$  $A$ is unchanged and $span\{\widetilde{e_1}\}$ is the common invariant space for  $A,B$.  Then by Lemma \ref{lem:basis}, we let $W=span\{\widetilde e_1\}=span\{\widetilde b_1\}$, we get $ABA=\begin{pmatrix}
\lambda_1&\beta_1&\gamma_1\\
0&\beta_2&\gamma_2\\
0&\beta_3&\gamma_3
\end{pmatrix}$ as usual and this time $\beta_3\neq 0$ ($\beta_3=0$ is the case that $\dim(W_0)=2$), and we can simplify the form by taking $\beta_3=1$, by conjugating by a diagonal matrix, we get \ref{n=3_W=1_A2B}.

\item \label{case:n=3_W=1_case3}:
$A$ has one two-dimensional Jordon block, and two distinct eigenvalues, without loss of generality, $\{e_2,e_3\}$ corresponds to eigenvalue $\lambda_2$ and $e_3$ is a generalized eigenvector. Here are two kinds of choices for $W$, one is
$W=span\{e_1\}=span\{b_1\}$, by same argument and computations we get \ref{n=3_W=1_A3B1}, \ref{n=3_W=1_A3B2}; the other choice is $W=span\{e_2\}=span\{b_2\}$, then we have \ref{n=3_W=1_A4B1}, \ref{n=3_W=1_A4B2}.

\item \label{case:n=3_W=1_case4}:
$A$ has one two dimensional Jordon block with only one eigenvalue $\lambda$, then without loss of generality, $W_1=span\{e_1,e_2\}$ is the two dimensional eigenspace for $A$ with respect to $\lambda$, $W_2=span\{b_1,b_2\}$ is the two dimensional eigenspace with respect to $\lambda$ for $B$, Let $W_0=W_1\cap W_2$, then by similar argument in \ref{case:n=3_W=1_case2}, when $\dim(W_0)=2$, we can reduce $B$ to the form $\begin{pmatrix}
\lambda&0&1\\
0&\lambda&\gamma\\
0&0&\lambda
\end{pmatrix}$, by easy computation, this case is impossible. When $\dim(W_0)=1$, we have a similar story with \ref{case:n=3_W=1_case2}, except the case when $t_1=0$; when $t_1\neq 0$, we follow the trick and reduce it to the case that $W=span\{e_1,b_1\}$, similarly we get a decomposable representation; when $t_1=0$, again by Lemma \ref{lem:basis}, we know $W=span\{e_2,b_2\}$, then we get $ABA=\begin{pmatrix}
\alpha_1&0&\gamma_1\\
\alpha_2&\lambda^3&\gamma_2\\
\alpha_3&0&\gamma_3
\end{pmatrix}$, by solving the equation we get $\gamma_1=\alpha_2=\alpha_3=0$, then $V=span\{e_1\}\oplus span\{e_2,e_3\}$, $A,B$ are both invariant on each summand, which means it is decomposable as well. 
\item \label{case:n=3_W=1_case5}:
$A$ is a 3-dimensional Jordan block, the only choice for $W$ is $span\{e_1\}=span\{b_1\}$, again by using Maple we get \ref{n=3_W=1_A5B}.
\end{enumerate}

When $A$ is a scalar matrix, by conjugating $M=ABA$, we get $B=A$, which is obviously decomposable. 

Finally, we can easily verify all the solutions we get are distinct and indeed indecomposable but not irreducible (see Appendix), which finishes the proof. 
\end{proof}

\begin{theorem}\label{thm:n=3_W=2}
when $n=3,\  \dim(W)=2$, we have following results: ($\lambda_i\neq 0$ are distinct)

\begin{enumerate}[label=\upshape(\arabic*)
]
\item\label{n=3_W=2_A1B1}
$\ A=\begin{pmatrix}
\lambda_1&0&0\\
0&\lambda_2&0\\
0&0&\lambda_3
\end{pmatrix}$  \ \ \ \ \  
$\ B=\begin{pmatrix}
\frac{\lambda_2^3}{\lambda_2^2+\lambda_3^2}&1&1\\
\frac{\lambda_3^2(\lambda_2^4+\lambda_2^2\lambda_3^2+\lambda_3^4)}{(\lambda_2^2+\lambda_3^2)^2}&-\frac{\lambda_3^4}{\lambda_2(\lambda_2^2+\lambda_3^2)}&-\frac{\lambda_3^3(\lambda_2^2+\lambda_2\lambda_3+\lambda_3^2)}{\lambda_2^2(\lambda_2^2+\lambda_3^2)}\\
0&0&\lambda_3
\end{pmatrix}$\\ 
($\lambda_1=-\frac{\lambda_3^2}{\lambda_2} \ \ \text{and}\ \ \lambda_1^2-\lambda_1\lambda_2+\lambda_2^2\neq 0$)

Or equivalently  
$\ A=\begin{pmatrix}
\lambda_2&\lambda_2&0\\
0&\lambda_1&0\\
0&0&\lambda_3
\end{pmatrix}$  \ \ \ \ \  
$\ B=\begin{pmatrix}
\lambda_1&0&\frac{\lambda_3^2-(\lambda_1-\lambda_2)\lambda_3+\lambda_2^2}{\lambda_2^2+\lambda_3^2}\\
-\lambda_1&\lambda_2&-1\\
0&0&\lambda_3
\end{pmatrix}$

\item\label{n=3_W=2_A2B2}
$\ A=\begin{pmatrix}
\lambda_1&1&0\\
0&\lambda_1&0\\
0&0&\lambda_2
\end{pmatrix}$  \ \ \ \ \  
$\ B=\begin{pmatrix}
2\lambda_1&1&\frac{2\lambda_1+\lambda_2}{\lambda_2^2}\\
\lambda_2^2&0&1\\
0&0&\lambda_2
\end{pmatrix}$
\ \ \ ($\lambda_1=\lambda_2 e^{\pm\pi i/2}$)

Or equivalently
$\ A=\begin{pmatrix}
\lambda_1&\lambda_1&0\\
0&\lambda_1&0\\
0&0&\lambda_2
\end{pmatrix}$  \ \ \ \ \  
$\ B=\begin{pmatrix}
\lambda_1&0&1+\frac{\lambda_1}{\lambda_2}\\
-\lambda_1&\lambda_1&\frac{\lambda_2}{\lambda_1}\\
0&0&\lambda_2
\end{pmatrix}$

\item\label{n=3_W=2_A3B3}
$\ A=\begin{pmatrix}
\lambda_1&0&0\\
0&\lambda_2&1\\
0&0&\lambda_2
\end{pmatrix}$  \ \ \ \ \  
$\ B=\begin{pmatrix}
\frac{\lambda_2}{2}&1&0\\
\frac{3\lambda_2^2}{4}&-\frac{\lambda_2}{2}&-2\\
0&0&\lambda_2
\end{pmatrix}$ \ \ \ \ \ ($\lambda_1=-\lambda_2$)

Or equivalently
$\ A=\begin{pmatrix}
\lambda_1&\lambda_1&1\\
0&\lambda_2&-2\\
0&0&\lambda_2
\end{pmatrix}$  \ \ \ \ \  
$\ B=\begin{pmatrix}
\lambda_2&0&-2\\
-\lambda_2&\lambda_1&4\\
0&0&\lambda_2
\end{pmatrix}$
\end{enumerate}
\begin{remark}
    Notice that in all cases the polynomial \[(\lambda_1^2+\lambda_2\lambda_3)(\lambda_2^2+\lambda_1\lambda_3)(\lambda_3^2+\lambda_1\lambda_2)\] vanishes so that no 3-dimensional $B_3$ representation with these eigenvalues can be irreducible.
\end{remark}
\end{theorem}
\begin{proof}
As usual, we first reduce $A$ to the Jordan normal form.  Observe that in each of the cases \ref{case:n=3_W=1_case2}, \ref{case:n=3_W=1_case4} and \ref{case:n=3_W=1_case5} there is a 1-dimensional invariant subspace, contrary to the assumption $\dim(W)=2$. Therefore it suffices to consider the following two cases:
\begin{enumerate}[label=\upshape \textbf{Case\ \arabic*}, itemindent=3.5em, leftmargin=0pt]
 \item \label{case:n=3_W=2_case1}:
$A$ is a diagonal matrix with three different eigenvalues.  In this case the 2-dimensional invariant space for $A$ can only be $span\{e_i,e_j\}\;i,j\in\{1,2,3\}$.  Without loss of generality, we may assume $W=span\{e_1,e_2\}$ is the common invariant subspace of $A,B$.  By Lemma \ref{lem:basis}, $W=span\{b_1,b_2\}$, then we get $ABA=\begin{pmatrix}
\alpha_1&\beta_1&\gamma_1\\
\alpha_2&\beta_2&\gamma_2\\
0&0&\gamma_3
\end{pmatrix}$, with $\alpha_i,\beta_i\neq 0, i\in \{1,2\}$, otherwise $\dim(W)=1$, and $\gamma_3\neq 0$,  by similar computation we get \ref{n=3_W=2_A1B1}.
 \item \label{case:n=3_W=2_case2}: $A$ has one two-dimensional Jordon block, and two distinct eigenvalues, then the 2-dimensional invariant space for $A$ can only be the generalized eigenspace or the direct sum of two 1-dimensional eigenspace, without loss of generality, we let $span\{e_1,e_2\}$ be the generalized eigenspace, then for one possibility, $W=span\{e_1,e_2\}$, and $W=span\{b_1,b_2\}$ (Lemma \ref{lem:basis}), therefore we get the same form of $ABA$ as in \ref{case:n=3_W=1_case1}, except that $\alpha_1$ or $\beta_1$ could be 0 this time, then by computation, we get \ref{n=3_W=2_A2B2}. For the other possibility of $W$, let $W=span\{e_1,e_3\}$, thus $W=span\{b_1,b_3\}$, by suitably adjusting the basis, we obtain the same form of $ABA$ as in \ref{case:n=3_W=2_case1}, finally we get \ref{n=3_W=2_A3B3}.
\end{enumerate}

It's easy to verify the solutions are indecomposable and distinct  since $A$ satisfies that its characteristic polynomial and minimal polynomial are equal. 
\end{proof} 

\section{Applications and Future Work}

Analyzing the braid group representations associated with non-semisimple braided categories is an immediate application to our classification.  Another key area in which indecomposable braid group representations appear is in the study of Yang-Baxter operators.  

\textbf{Application.}
Consider the following two Yang-Baxter matrices, which are homogeneous versions of the $R$-matrices associated with the quantum groups $U_q\mathfrak{sl}_2$ and $U_q\mathfrak{gl}(1|1)$.  Here $a,b$ are distinct nonzero complex numbers.

\[
R_1=\begin{pmatrix}
a&0&0&0\\
0&a+b&-b&0\\
0&a&0&0\\
0&0&0&a
\end{pmatrix},\ \   R_2=\begin{pmatrix}
a&0&0&0\\
0&a+b&-b&0\\
0&a&0&0\\
0&0&0&b
\end{pmatrix}
\]
These yield 8-dimensional representations for $B_3$ by 
\[
A_i=\rho_i(\sigma_1)=R_i\otimes \begin{pmatrix}
1&0\\
0&1
\end{pmatrix},\ \ \ B_i=\rho_i(\sigma_1)=\begin{pmatrix}
1&0\\
0&1
\end{pmatrix}\otimes R_i 
\]
If $a^2-ab+b^2\neq 0$, and also, for $i=2$,  \[(2a^2 - ab - a + b)\big((6b - 1)a^2 + (-8b^2 + b)a + 4b^3\big)(a^2 - ab - b^2)\neq 0\] these representations are completely reducible, with the following irreducible decompositions:
\[A_1= a\oplus a \oplus a\oplus a\oplus \begin{pmatrix}
a&a\\
0&b
\end{pmatrix} \oplus \begin{pmatrix}
a&a\\
0&b
\end{pmatrix},\ \  B_1= a\oplus a \oplus a\oplus a\oplus \begin{pmatrix}
b&0\\
-b&a
\end{pmatrix} \oplus \begin{pmatrix}
b&0\\
-b&a
\end{pmatrix}
\]
\[
A_2=a\oplus a \oplus b\oplus b\oplus \begin{pmatrix}
a&a\\
0&b
\end{pmatrix} \oplus \begin{pmatrix}
a&a\\
0&b
\end{pmatrix},\ \  B_1= a\oplus a \oplus b\oplus b\oplus \begin{pmatrix}
b&0\\
-b&a
\end{pmatrix} \oplus \begin{pmatrix}
b&0\\
-b&a
\end{pmatrix}
\]
When $a,b$ satisfy $a^2-ab+b^2=0$ and, for $i=2$, $(12b^2 - 6b + 1)(b^2+a-b-2ab)\neq0$, we have the following decomposition into indecomposables:
\[
A_1=a\oplus a \oplus 
\begin{pmatrix}
a&0&0\\
0&a&0\\
0&0&b
\end{pmatrix}\oplus
\begin{pmatrix}
a&0&0\\
0&a&0\\
0&0&b
\end{pmatrix},\ \ B_1=a\oplus a \oplus 
\begin{pmatrix}
a&1&b-a\\
0&a&0\\
0&1&b
\end{pmatrix}\oplus
\begin{pmatrix}
a&1&b-a\\
0&a&0\\
0&1&b
\end{pmatrix}
\]
\[
A_2=a\oplus b \oplus 
\begin{pmatrix}
a&0&0\\
0&a&0\\
0&0&b
\end{pmatrix}\oplus
\begin{pmatrix}
b&0&0\\
0&b&0\\
0&0&a
\end{pmatrix},\ \ B_2=a\oplus b \oplus 
\begin{pmatrix}
a&1&b-a\\
0&a&0\\
0&1&b
\end{pmatrix}\oplus
\begin{pmatrix}
b&1&a-b\\
0&b&0\\
0&1&a
\end{pmatrix}
\]
The two distinct 3-dimensional indecomposable representations are both of the form found in Theorem \ref{thm:n=3_W=1} \ref{n=3_W=1_A2B}.

\textbf{Extensions} It is natural to consider classifying indecomposable but not irreducible representations in higher dimensions.  One way is to simply extend our method, which is feasible but much more complicated.
The other way is to use the result in paper \cite{TW01} about irreducible representations we can always reduce $A, B$ to have the following form:
$
A=\begin{pmatrix}
A_1 &\star\\
0&A_2
\end{pmatrix}
$, 
$
B=\begin{pmatrix}
B_1 &\star\\
0&B_2
\end{pmatrix}
$, where $(A_1,B_1)$ gives an irreducible representation.  Although we know their forms from \cite{TW01}, it is inefficient to keep all the unknown variables and directly calculate the equations given by the braid relation. Thus one needs to do some analysis for the forms of $A_2$ and $B_2$ to simplify the calculations. Constructing strictly indecomposable representation directly is also

\section{Appendix}\label{sec:appendix}

Here we give more details of the proof of Theorem \ref{thm:n=3_W=1}, \ref{thm:n=3_W=2}. The key algorithm we use in Maple is to find a Groebner basis.  With a fixed Jordan form of $A$, we start with the polynomial entries of $ABA=BAB$, which are in terms of the eigenvalues of $A$ and the undetermined parameters of $ABA$.  We enforce non-vanishing of the eigenvalues as well the requirement that they be distinct by means of a single polynomial with a new variable $k$, e.g. $\lambda_1\lambda_2(\lambda_1-\lambda_2)k-1$.  We use an elimination order to find a set of polynomial consequences that do not use the variable $k$, and then run the Groebner basis algorithm again with a pure lexicographical order to simplify the equations, which can be analyzed by hand. 

To illustrate we will give detailed proof for the case $\dim(W)=1$ and $A$ is a diagonal matrix with three distinct eigenvalues (Theorem \ref{thm:n=3_W=1}, \ref{case:n=3_W=1_case1}).  We begin with the form $$A=\begin{pmatrix}
\lambda_1&0&0\\
0&\lambda_2&0\\
0&0&\lambda_3
\end{pmatrix},\ \ 
ABA=\begin{pmatrix}
\lambda^3_1&\beta_1&\gamma_1\\
0&\beta_2&\gamma_2\\
0&\beta_3&\gamma_3
\end{pmatrix}.$$

\begin{enumerate}[label=\upshape \textbf{Case\ 1.\arabic*}, itemindent=4.5em, leftmargin=0pt]
 \item \label{case:1.1}: If
$\gamma_2=\beta_3=0$, then by indecomposability and invertibility, all other variables are nonzero, and we may assume $\beta_1=\gamma_1=1$, by conjugating the matrix $\begin{pmatrix}
1&0&0\\
0&\beta_1&0\\
0&0&\gamma_1
\end{pmatrix}$.
We get a set of polynomials from equation $ABA=BAB$.  Using the Groebner basis algorithm in Maple we get the following set of new polynomials that must vanish:
$$
\lambda_1-\lambda_2-\lambda_3,\  \beta_2-\gamma_3,\ \lambda_2^2+\lambda_2\lambda_3+\lambda_3^2,\ \lambda_2^3-\gamma_3
.$$
Solving them, we obtain:

$
B=\begin{pmatrix}
\lambda_1&\frac{1}{\lambda_1\lambda_2}&\frac{1}{\lambda_1\lambda_3}\\
0&\lambda_2&0\\
0&0&\lambda_3
\end{pmatrix}
$.  Conjugating by 
$
\begin{pmatrix}
1&0&0\\
0&\frac{1}{\lambda_1\lambda_2}&0\\
0&0&\frac{1}{\lambda_1\lambda_3}
\end{pmatrix}
$, we have \ref{n=3_W=1_A1B1}.

\item \label{case:1.2}: If 
$\gamma_2\neq 0,\beta_1=\beta_3=0$, again by indecomposability and invertibility all other variables are nonzero. Hence by a similar argument, one can assume $\gamma_1=\gamma_2=1$. Using Maple we get a similar set:
$$
\lambda_1+\lambda_2-\lambda_3,\  \beta_2+\gamma_3,\ \lambda_2^2-\lambda_2\lambda_3+\lambda_3^2,\ \lambda_2^3-\gamma_3.
$$
Solve them we get \ref{n=3_W=1_A1B2}.

\item \label{case:1.3}:
If $\gamma_2\neq 0,\beta_1\neq0,\beta_3=0$, then $\gamma_3\neq 0$. Using Maple we get the following set:
$$
\lambda_1-\lambda_2+\lambda_3,\  \beta_2+\gamma_3,\ \lambda_2^2-\lambda_2\lambda_3+\lambda_3^2,\ \beta_1\gamma_2+2\gamma_1\gamma_3,\ \lambda_2^3-\gamma_3.
$$
We have $\gamma_1\neq 0$ from $\beta_1\gamma_2+2\gamma_1\gamma_3=0$, hence we set $\gamma_1=\beta_1=1$. 
Using the Groebner command again (or directly plugging in the previous set), we get:
$$
\lambda_1-\lambda_2+\lambda_3,\  \beta_2+\gamma_3,\ \lambda_2^2-\lambda_2\lambda_3+\lambda_3^2,\ \gamma_2+2\gamma_3,\ \lambda_2^3-\gamma_3.
$$
Solving them, we have \ref{n=3_W=1_A1B3}.
\item \label{case:1.4}:  If
$\gamma_2,\beta_3\neq 0$, then we get a longer list of polynomials from the Groebner basis command. The following are five of them:

\begin{align*}
&\beta_2+\gamma_3,\  \gamma_1(\lambda_1^2+\lambda_2\lambda_3),\ \beta_1(\lambda_1^2+\lambda_2\lambda_3),\\
&\lambda_1\beta_1\gamma_2+\lambda_1\gamma_1\gamma_3+\lambda_2\gamma_1\gamma_3-\lambda_3\gamma_1\gamma_3,\ -\lambda_1\beta_1\gamma_3+\lambda_1\beta_3\gamma_1+\lambda_2\beta_1\gamma_3-\lambda_3\beta_1\gamma_3.
\end{align*}
Together with indecomposability ($(\gamma_1,\beta_1)\neq (0,0)$) implies $\lambda_1^2+\lambda_2\lambda_3=0$ and all the variables are nonzero, moreover $\lambda_1^2+\lambda_2^2\neq 0$, since $\lambda_2\neq\lambda_3$. Now we can set $\beta_1=\beta_3=1$, The Groebner basis command gives a shorter list, it turns out that the following first five relations are enough to get the solution:
$$
\beta_2+\gamma_3,\ \lambda_3\lambda_2+\lambda_1^2,\ \lambda_1\gamma_1-\lambda_1\gamma_3+\lambda_2\gamma_3-\lambda_3\gamma_3,\ \gamma_1^2-2\gamma_1\gamma_3-\gamma_2,\ \lambda_2\lambda_3\lambda_1-\gamma_1+\gamma_3.
$$
Solving them we have \ref{n=3_W=1_A1B4}.
\end{enumerate}

Since by permuting the basis $e_2$ and $e_3$, one can switch the roles played by $\gamma_2$ and $\beta_3$, \ref{case:1.2} and \ref{case:1.3} give all the solutions for the cases when one of $\gamma_2,\beta_3$, is zero or non-zero hence \ref{case:1.1}-\ref{case:1.4} give all the possible solutions when $A$ is a diagonal matrix with three distinct 
eigenvalues. 

Moreover, only diagonal matrices commute with $A$, and conjugating $B$ by a diagonal matrix won't change the positions where the entries are zero or non-zero, therefore the distinctness follows easily.  As for indecomposability, since the eigenvalues are all distinct, the only possible decompositions will be  $Ke_i\oplus K\{e_j,e_k\} $ for distinct $i,j,k$ and $1\leq i,j,k\leq 3$, which is impossible for $B$.

Now for \ref{case:n=3_W=1_case2} in Theorem \ref{thm:n=3_W=1}, we begin with the following matrix

$$A=\begin{pmatrix}
\lambda_1&0&0\\
0&\lambda_1&0\\
0&0&\lambda_2
\end{pmatrix},\ \ 
ABA=\begin{pmatrix}
\lambda^3_1&\beta_1&\gamma_1\\
0&\beta_2&\gamma_2\\
0&1&\gamma_3
\end{pmatrix}.$$

\begin{enumerate}[label=\upshape \textbf{Case\ 2.\arabic*}, itemindent=4.5em, leftmargin=0pt]
 \item \label{case:2.1}:
 $\gamma_2\neq 0$, using the Groebner basis command with, in addition to $ABA=BAB$, relation $\gamma_2(\lambda_1 - \lambda_2)\lambda_1\lambda_2(\beta_2\gamma_3 - \gamma_2)K - 1$ added, we have the following first several terms:
 \[
\beta_2+\gamma_3, \ \lambda_2\beta_1\gamma_3 - \lambda_1\gamma_1,\  \lambda_1^2\lambda_2^2-\lambda_1\gamma_3 + \lambda_2\gamma_3,\  \lambda_1\lambda_2^2\beta_1 - \beta_1\gamma_3 + \gamma_1.
 \]
 From the third term, we have $\gamma_3\neq 0$, which implies $\beta_2\neq 0$ from the first term. Moreover, from the second term $\beta_1=0\Leftrightarrow\gamma_1\neq 0$, by indecomposability they are both nonzero, hence we can assume $\beta_1=1$. Update our set of relations and use the Groebner basis command again, we have a list of relations, the following are the first four terms:
 \[
 \beta_2 + \gamma_3,\  \lambda_1\gamma_1 - \lambda_2\gamma_3,\  \lambda_1\lambda_2^2 + \gamma_1 - \gamma_3,\  \lambda_1\gamma_3^2 + \lambda_2\gamma_1\gamma_3 - \lambda_2\gamma_3^2 + \lambda_2\gamma_2.
 \]
 Which gives the solution: $
 B=\begin{pmatrix}
\lambda_1&\frac{1}{\lambda_1^2}&\frac{\lambda_2^2}{\lambda_1-\lambda_2}\\
0&\frac{-\lambda_2^2}{\lambda_1-\lambda_2}&\frac{\lambda_1^2\lambda_2^2(\lambda_1^2-\lambda_1\lambda_2+\lambda_2^2)}{(\lambda_1-\lambda_2)^2}\\
0&\frac{1}{\lambda_1\lambda_2}&\frac{\lambda_1^2}{\lambda_1-\lambda_2}
\end{pmatrix}$, our assumption $\gamma_2\neq 0$ implies $\lambda_1^2-\lambda_1\lambda_2+\lambda_2^2\neq 0$.
Conjugating $B$ by 
$\begin{pmatrix}
1&1&0\\
0&\frac{\lambda_1^2(\lambda_1^2-\lambda_1\lambda_2+\lambda_2^2)}{\lambda_1-\lambda_2}&0\\
0&0&1
\end{pmatrix}$ (it commutes with $A$), we see the solution is decomposable. (Equivalently, it decomposes as $Ke_1\oplus K\{e_2+\frac{\lambda_1^2(\lambda_1^2-\lambda_1\lambda_2+\lambda_2^2)}{\lambda_1-\lambda_2}e_3,e_3\}$)

\item \label{case:2.2}: $\gamma_2=0$, similarly, we have several terms (part of the list):
\[
\beta_2 + \gamma_3,\  \lambda_1^2 - \lambda_1\lambda_2 + \lambda_2^2,\  \lambda_2^3 - \gamma_3,\  \lambda_2\beta_1\gamma_3 - \lambda_1\gamma_1.
\]
 \end{enumerate}
This implies $\gamma_1,\gamma_3,\beta_1,\beta_2$ are all nonzero, again we assume $\beta_1=1$. Then solve a short list of equations, we have \ref{n=3_W=1_A2B}, its indecomposability will be discussed later.

There are two types of $W$ in \ref{case:n=3_W=1_case3} of Theorem \ref{thm:n=3_W=1}, First one gives the following representation:
$$A=\begin{pmatrix}
\lambda_1&0&0\\
0&\lambda_2&1\\
0&0&\lambda_2
\end{pmatrix},\ \ 
ABA=\begin{pmatrix}
\lambda^3_1&\beta_1&\gamma_1\\
0&\beta_2&\gamma_2\\
0&\beta_3&\gamma_3
\end{pmatrix}.$$

\begin{enumerate}[label=\upshape \textbf{Case\ 3.\arabic*}, itemindent=4.5em, leftmargin=0pt]
 \item \label{case:3.1}:
 $\beta_3=0$, using the Groebner basis command, we directly get $b_1=0$ and $\lambda^2_1-\lambda_1\lambda_2+\lambda^2_2=0$, then \ref{n=3_W=1_A3B1} is obtained by simple calculations.
 \item \label{case:3.2}:
 $\beta_3\neq 0$, using the Groebner basis command, the following is a subset of the terms we obtain:
 \[
 \gamma_1(\lambda_1^2+\lambda_2^2),\ \beta_1(\lambda_1^2+\lambda_2^2),\  -\lambda_1\beta_1\gamma_3 + \lambda_1\beta_3\gamma_1 - \beta_1\beta_3.
 \]
The second and third terms together with the indecomposability imply $\lambda_1^2+\lambda_2^2=0$. The last term together with the indecomposability implies $\beta_1\neq 0$, hence we assume $\beta_1=1$. Now we simply the set of relations and we obtain the following subset,
\begin{eqnarray*}
   \beta_2 + \gamma_3,\  \lambda_1^2+\lambda_2^2, &&\  \lambda_1\lambda_2^2 - \gamma_1\beta_3 + \gamma_3,\  \lambda_2^2\beta_3 - \beta_3\gamma_2 - \gamma_3^2,\\&& \lambda_1\gamma_1\gamma_3 + \lambda_1\gamma_2 - \beta_3\gamma_1,\  \gamma_1\beta_3\lambda_1 - \gamma_3\lambda_1 - \beta_3.
\end{eqnarray*}

This is enough to get the solution:
\[B=\begin{pmatrix}
\lambda_1&\frac{1}{\lambda_1\lambda_2}&\frac{1}{\lambda^3_1}-\frac{1}{\lambda_2^3}-\frac{\gamma_3}{\lambda_1^5\lambda_2}\\
0&\lambda_2-\frac{\gamma_3}{\lambda_2^2}&\frac{\gamma_3}{\lambda^4_1\lambda^2_2}\\
0&\lambda^2_1&\lambda_2+\frac{\gamma_3}{\lambda_2^2}
\end{pmatrix}.
\]
Conjugating by  $\begin{pmatrix}
\frac{1}{\lambda_1\lambda_2}&0&0\\
0&1&-\frac{\gamma_3}{\lambda^2_1\lambda^2_2}\\
0&0&1
\end{pmatrix}
$ (it commutes with $A$), we have \ref{n=3_W=1_A3B2}.

\item\label{case:3.3}:
The second type of $W$ gives the following representation:
$$A=\begin{pmatrix}
\lambda_1&1&0\\
0&\lambda_1&0\\
0&0&\lambda_2
\end{pmatrix},\ \ 
ABA=\begin{pmatrix}
\lambda^3_1&\beta_1&\gamma_1\\
0&\beta_2&\gamma_2\\
0&\beta_3&\gamma_3
\end{pmatrix}.$$
We first assume $\gamma_2\neq 0$, and up to equivalence, we may assume $\gamma_2=1$. Applying the Groebner basis command, we obtain the following short list of relations,
\[
\gamma_1+\gamma_2,\ \beta_2 + \gamma_3,\  3\gamma_3^2 - \beta_3,\  3\gamma_1\gamma_3 + \beta_1,\  \beta_1\gamma_3 + \beta_3\gamma_1,\  \lambda_2^3 + 2\gamma_3,
\]
which gives the solution $B=\begin{pmatrix}
-\lambda_2&\frac{3\gamma_1\lambda_2-1}{2}&\frac{-\gamma_1\lambda_2-1}{\lambda_2^3}\\
0&\frac{\lambda_2}{2}&-\frac{1}{\lambda^2_2}\\
0&-\frac{3\lambda^4_2}{4}&-\frac{\lambda_2}{2}
\end{pmatrix}$.  Conjugating by  $\begin{pmatrix}
-\frac{1}{\lambda^2_2}&-\frac{\gamma_1\lambda_2+1}{\lambda^3_2}&0\\
0&-\frac{1}{\lambda^2_2}&0\\
0&0&1
\end{pmatrix}$, we obtain \ref{n=3_W=1_A4B1}.

\item \label{case:3.4}: 
$\gamma_2=0$, this implies $\gamma_1\neq 0$, otherwise $Ke_3$ will be a 1-dimensional common invariant subspace, which is the $W$ of the first type in \ref{case:n=3_W=1_case3} (see \ref{case:3.1}, \ref{case:3.2}). Therefore, we assume $\gamma_1=1$.  Applying the Groebner basis command, we have the following subset:
\[
\beta_2 + \gamma_3,\  \lambda_2^2 - \lambda_1\lambda_2 + \lambda_1^2,\  \lambda_2^3 - \gamma_3,\  \lambda_1\beta_1\gamma_3 - \lambda_1\beta_3 + \lambda_2\beta_3 + 3\gamma_3^2.
\]
Solving these equations gives the solution: $B=\begin{pmatrix}
\lambda_1&-2+\frac{\beta_1}{\lambda_1^2}&\frac{1}{\lambda_1\lambda_2}\\
0&\lambda_1&0\\
0&3\lambda^3_1-\beta_1\lambda_1&-\frac{\lambda^3_1}{\lambda_2^2}
\end{pmatrix}$.  By simply changing the basis $\tilde{e}_3:=\lambda_1\lambda_2e_3$ and setting $\beta:=\frac{3\lambda^2_1-\beta_1}{\lambda_2}$, we obtain \ref{n=3_W=1_A4B2}.
\end{enumerate}

 Other cases in Theorem \ref{thm:n=3_W=1}, \ref{thm:n=3_W=2} are less complicated and follow similar procedures, so we will now turn to proving our solutions are distinct and indecomposable.
 
 \textbf{Distinctness}: First note that the representations in Theorem \ref{thm:n=3_W=1} and Theorem \ref{thm:n=3_W=2} are inequivalent, and that different Jordan forms of $A$ give inequivalent representations. Moreover, except \ref{n=3_W=1_A4B1}, all other representations depend only on eigenvalues and relations among them. Hence only \ref{n=3_W=1_A3B1} and \ref{n=3_W=1_A4B1} need to be analyzed.
 
 In \ref{n=3_W=1_A3B1}, eigenvectors corresponding to two eigenvalues are two $1$-dimensional common invariant subspaces, which is not the case in \ref{n=3_W=1_A4B1}. Moreover different $\beta$ give different representations, in fact, all the matrices commuting with $A$ in \ref{n=3_W=1_A4B1} are of the form: $M=\begin{pmatrix}
 a&b&0\\
 0&a&0\\
 0&0&c
 \end{pmatrix}
 $.  Direct calculation shows the $\beta$ term is invariant under conjugating $B$ by $M$.
 
 \textbf{Indecomposability}: Except from \ref{n=3_W=1_A2B} in Theorem \ref{thm:n=3_W=1}, the invariant subspaces decomposition of $A$ in all other cases have only finitely many choices, since all eigenspaces are one dimensional. It is not hard to see they are indecomposable. For \ref{n=3_W=1_A2B}, since by construction, $Ke_1$ is a common invariant subspace and $K\{e_1,e_2\}$ is not, we cannot have any other one-dimensional common invariant subspaces. Therefore if the representation is decomposable, it can only decompose as $Ke_1\oplus K\{ae_1+be_2,e_3\}$ for some $b\neq 0$, which is impossible for $B$.

 \bibliographystyle{spmpsci}
 \bibliography{Reference}

\begin{thebibliography}{1}
\providecommand{\url}[1]{{#1}}
\providecommand{\urlprefix}{URL }
\expandafter\ifx\csname urlstyle\endcsname\relax
  \providecommand{\doi}[1]{DOI~\discretionary{}{}{}#1}\else
  \providecommand{\doi}{DOI~\discretionary{}{}{}\begingroup
  \urlstyle{rm}\Url}\fi

\bibitem{KT08}
Kassel, C., Turaev, V.: Braid groups, \emph{Graduate Texts in Mathematics},
  vol. 247.
\newblock Springer, New York (2008).
\newblock \doi{10.1007/978-0-387-68548-9}.
\newblock \urlprefix\url{https://doi.org/10.1007/978-0-387-68548-9}.
\newblock With the graphical assistance of Olivier Dodane

\bibitem{LRW}
Larsen, M.J., Rowell, E.C., Wang, Z.: The {$N$}-eigenvalue problem and two
  applications.
\newblock Int. Math. Res. Not. (64), 3987--4018 (2005).
\newblock \doi{10.1155/IMRN.2005.3987}.
\newblock \urlprefix\url{https://doi.org/10.1155/IMRN.2005.3987}

\bibitem{Maple}
{Maplesoft, a division of Waterloo Maple Inc..}: Maple.
\newblock \urlprefix\url{https://hadoop.apache.org}

\bibitem{ruma}
Rowell, E.C.: Braid representations from quantum groups of exceptional {L}ie
  type.
\newblock Rev. Un. Mat. Argentina \textbf{51}(1), 165--175 (2010)

\bibitem{RW18}
Rowell, E.C., Wang, Z.: Mathematics of topological quantum computing.
\newblock Bull. Amer. Math. Soc. (N.S.) \textbf{55}(2), 183--238 (2018).
\newblock \doi{10.1090/bull/1605}.
\newblock \urlprefix\url{https://doi-org.libezp.lib.lsu.edu/10.1090/bull/1605}

\bibitem{TW01}
Tuba, I., Wenzl, H.: Representations of the braid group {$B_3$} and of {${\rm
  SL}(2,{\bf Z})$}.
\newblock Pacific J. Math. \textbf{197}(2), 491--510 (2001).
\newblock \doi{10.2140/pjm.2001.197.491}.
\newblock \urlprefix\url{https://doi.org/10.2140/pjm.2001.197.491}

\end{thebibliography}


\begin{thebibliography}{9}
\bibitem{braid group} 
Christian Kassel, Vladimir Turaev.
\textit{Braid groups}. 
Springer science, 2008.
 
\bibitem{REPRESENTATIONS OF THE BRAID GROUP $B_3$ AND OF $SL(2,Z)$} 
Imre Tuba, Hans Wenzl.
\textit{REPRESENTATIONS OF THE BRAID GROUP $B_3$ AND OF $SL(2,Z)$}
arXiv: math 9912013v1, 1999.

\bibitem{An Invitation to the Mathematics of Topological Quantum Computation}
Eric C. Rowell
\textit{An Invitation to the Mathematics of Topological Quantum Computation}
Journal of Physics: Conference Series 698 (2016) 012012.
\bibitem{Representation Theory}
W. Fulton, J. Harris
\textit{Representation Theory}
Springer science, 1991
\end{thebibliography}

\end{document}